\documentclass[12pt]{amsart}
\usepackage{amssymb,verbatim,enumerate,ifthen}
\usepackage[mathscr]{eucal}
\usepackage[utf8]{inputenc}
\usepackage[T1]{fontenc}

\makeatletter
\@namedef{subjclassname@2010}{%
  \textup{2010} Mathematics Subject Classification}
\makeatother

\newtheorem{prop}{Proposition}

\numberwithin{equation}{section}

\newcommand\nolabel[1]{\nonumber}

\newcommand\R{\mathbb{R}}

\newcommand\N{\mathbb{N}}
\newcommand\Q{\mathbb{Q}}
\newcommand\Z{\mathbb{Z}}
\newcommand\GI[2]{{[#1]_{#2}}}
\newcommand\id{\mathrm{id}}

\newcommand{\G}[1]{\mathcal{G}_{#1}}

\newcommand{\srpot}[1]{\mathcal{P}_{#1}}

\newcommand{\abs}[1]{\left| #1 \right| }

\numberwithin{equation}{section}

\def\Eq#1#2{\ifthenelse{\equal{#1}{*}}
  {\begin{equation*}\begin{aligned}[]#2\end{aligned}\end{equation*}}
  {\begin{equation}\begin{aligned}[]\label{#1}#2\end{aligned}\end{equation}}}

  \subjclass[2010]{26E60, 26D15}
\keywords{means, squeeze theorem, sandwich theorems, distance between means, Hardy means}

\title{Interval-type theorems concerning means}
\author[P. Pasteczka]{Pawe\l{} Pasteczka}
\address{Institute of Mathematics \\ Pedagogical University of Cracow \\ Podchor\k{a}\.zych str. 2, 30-084 Krak\'ow, Poland}
\email{ppasteczka@up.krakow.edu.pl}
\begin{document}

\begin{abstract} 
Each family $\mathcal{M}$ of means has a natural, partial order (point-wise order), that is $M \le N$ iff $M(x) \le N(x)$ for all admissible $x$.  

In this setting we can introduce the notion of interval-type set (a subset $\mathcal{I} \subset \mathcal{M}$ such that whenever $M \le  P \le N$ for some $M,\,N \in \mathcal{I}$ and $P \in \mathcal{M}$ then $P \in \mathcal{I}$). For example, in the case of power means there exists a natural isomorphism between interval-type sets and intervals contained in real numbers. Nevertheless there appear a number of interesting objects for a families which cannot be linearly ordered. 

In the present paper we consider this property for Gini means and Hardy means. Moreover some results concerning $L^\infty$ metric among (abstract) means will be obtained.
\end{abstract}

\maketitle

\section{Introduction}
It is well known that the comparability problem is one of the most extensively developed branch in the theory of means. In fact, whenever $Y$ is a family of means, we usually treat it as a partially ordered set. We will be interested in subfamilies $I \subset Y$ such that if some element of $Y$ is bounded from both sides by elements of $I$ then it itself belongs to $I$. More precisely, if for some $y \in Y$ there exists $y_l,y_u \in I$ such that $y_l \le y \le y_u$, then $y \in I$ too; such kind of condition is very characteristic for intervals, therefore we will call $I$ an \emph{interval-type set in $Y$}. 
  
The simplest examples of interval-type sets appear to be the natural generalization of intervals. Indeed, for every order-preserving embedding $Y \subset X$ (for example we can assign $Y=\Q$ and $X=\R$) and every $p,q \in X$, the set $[p,q]_Y:=\{y \in Y \colon p \le y \le q \}$ is an interval-type set in $Y$. 
Similarly, for $p \in X$, one can define $[p,+\infty)_Y:=\{y \in Y \colon p \le y\}$ etc; all these objects are interval-type sets in $Y$. 

Notice that if $\le$ is a linear order in $Y$ and $X=Y$ then these definitions coincide with 
the standard one. Moreover, if $I_1$ and $I_2$ are interval-type sets in $Y$ then so is $I_1 \cap I_2$, similarly a union of increasing sequence of interval-type sets in $Y$ is an interval-type set in $Y$ too. 
Finally if $Y \subseteq X$ and $p,q,r \in X$ with $p \le q \le r$ then $[p,q]_Y \cup [q,r]_Y \subseteq [p,r]_Y$ but, contrary to classic intervals, equality does not hold in a general case (even if $Y=X$). Furthermore it could be easily proved that if $I\subset Y\subset X$ then 
\Eq{ITSrest}{
\qquad &I \text{ is an interval-type set in }X\\
\Longrightarrow \quad I \cap &Y \text{ is an interval-type set in }Y.
}

In fact if $\le$ is a linear order then, roughly speaking, this definition reduces to common intervals.
The situation becomes much more interesting when $\le$ is just a partial ordering. For example if the inclusion is considered as the order, then each filter and each ideal is an interval-type set. 

From our point of view the most interesting set is a family of functions. There appears a natural order (point-wise order) that could be imprecisely define as $f \le g$ if and only if $f(x) \le g(x)$ for all $x$. In this sense all families denoted in the literature by $\mathcal{O}(\cdot)$ or $o(\cdot)$ are in fact interval-type sets. Additionally in the family of functions we have a number of heterogeneous interval-type sets like: functions that are convergent to a certain point (it is exactly what does squeeze theorem claim), bounded functions, $L^p$ spaces (in the family of all measurable functions).

This order is also used for comparing means -- this is in fact the setting we are heading towards. Indeed, comparability property is quite rare in the family of means, whence such a relation is a very natural partial order. Therefore from now on we are going to focus on a family of means only; remarkably we treat them simply as a functions and we do not assume any extra properties. Some families of means are mentioned in this paper just to provide both motivation and background of presented results -- for their precise definitions we refer the reader to the classical monography \cite{Bul03}.

Let me emphasize that there is no universal definition of the mean. Means are defined, in different places, for a various domain i.e. a vector of length two, vector of an arbitrary length, probabilistic measure, vector with weights (we usually use the prefix 'weighted' in this case); in fact {\it all} these definitions appear in \cite{Bul03}. Moreover there are a lot of axioms which are required, in different places, to call a function the mean -- the only which is beyond questioning is that the mean achieve an intermediate value between the minimal and maximal entry\footnote{This assumption was omitted for example by Matkowski in \cite{Mat14} but the name 'pre-mean` was used in that case.} and, surprisingly, continuity (usually it is not explicitly given as an axiom, however almost all means considered in the literature are continuous).

This situation leads to, sometimes confusing, misunderstandings. For example it is well known that so-called Gaussian product of arithmetic and harmonic means (known also as arithmetic-harmonic mean) is a geometric mean. However this equality holds only in a two variable setting. For more variables (in the spirit of Gustin \cite{Gus47}) arithmetic-harmonic mean equals a Gini mean $\mathcal{G}_{-1,1}$, which {\it coincide} with geometric one in a two variables case. On the other hand some means are defined from a certain number of parameters, for example some generalization of quasi-arithmetic means introduced by Sadikova \cite{Sad06} -- this definition was fulfilled, in a particular case, in \cite{Pas15b} to obtain some results concerning Hardy property. Alternatively the number of parameters can be estimated from above, for example in Hamy means and Hayashi means.

Despiting this drawback, there are usually no difficulties when it comes to define comparability between means. Namely, for two means (say $M$, $N$) we denote $M \le N$ if and only if $M(x)\le N(x)$ for all $x$ belonging to an {\it intersection} of domains of $M$ and $N$; such an assumption is so natural that it is usually skipped, however it is made even in a classical Cauchy's inequality (cf. \cite[p. 203]{Bul03}). On the other hand, power means $\srpot{1}$ and $\srpot{3}$ are not comparable when considered on reals. This causes essential problem that inequalities $M \le N$ and $N \le P$ do {\it not} imply the inequality $M \le P$. A very simple example is the pair $\srpot{1},\srpot{3}$ considered once on reals, and in the other case jointly with $\srpot{2}$ -- this time, to provide the meaningfulness of inequalities $\srpot{1}\le\srpot{2}$ and $\srpot{2}\le\srpot{3}$, all three means need to be considered on positive numbers only.

Due to this fact each time we are dealing with the comparability of means, we need to declare their domain.

Another domain-type problem can be observed when we look thorough the result concerning comparability in a family of Gini means.
In this family the solution of comparability problem changes dramatically while we are changing the number of arguments. Notice that, excluding the case of two- and an arbitrary number of arguments, this problem remains open (see \cite{Pal88c,Pal13ICFEI}). However, by the virtue of Jensen's inequality, this is not the case for quasi-arithmetic means. In this family comparability problems for vectors of length two, vectors of any (either fixed or arbitrary) length, their weighted counterparts and for measures are all equivalent among each other.

To continue dealing with means theory, let us make a short order-theory intermezzo concerning generalized intervals.
Let $(X,\le)$ and $(Y,\le)$ be two partially ordered sets (POSETs) such that either $X \subseteq Y$ or $Y \subset X$ and the order $\le$ coincide on $X \cap Y$. Denote
\Eq{*}{
\GI YX:=\{ x \in X \colon y_l \le x \le y_u \text{ for some } y_l,y_u \in Y\}=\bigcup_{y_l,y_u \in Y} [y_l,y_u]_X;
}
the order $\le$ in the definition is taken on $X \cup Y$.

In such case any consideration would go twofold, luckily the first case is simple and we are going to rule it out shortly. Indeed, for $X \subset Y$ we have
\Eq{*}{
X \subseteq X \cap Y \subseteq \GI YX \subseteq X.
}
Thus $\GI YX=X$ whenever $X\subseteq Y$. Therefore the only interesting case is $Y \subset X$.  Having this condition satisfied, it is natural to ask when $\GI YX$ is the smallest possible, that is $\GI YX =Y$. Whenever this equality holds we say that $Y$ is an \emph{interval-type in $X$} or, briefly, \emph{$X$-interval-type}.

Saying nothing of how common (or uncommon) this theory is, there are a number of interval-type sets among functional spaces (recall that in this setting $\le$ is a standard pointwise order). In fact we can just reformulate examples that are given above as interval-type sets, that is: functions that are converging to a certain (fixed) point, bounded functions, $L^p$ spaces (in a family of measurable functions).

In what follows we will investigate interval-type sets for various families of means. Some of results are either simple of just a new wording of known results. All these facts will be enclosed in the following section as well as new interval-type properties for means.


\section{Interval-type sets among means}

Many families of means has an order that is isomorphic to $(\N,\le)$ (for example symmetric polynomial means)
or $(\R,\le)$ (for example power means). Such kinds of families are, from our point of view, trivial cases. Obviously it does not mean that proving these inequalities is immediate. Conversely, these families are so well characterized that our theory, at least in my opinion, does not convey any additional knowledge in this case. At the moment each subsection will be devoted to different family of means. 

Let me emphasize that whenever $X$ is a family of means and $P ,Q \in X$ with $P \le Q$ then $[P,Q]_X$ are simply all intermediate means between $P$ and $Q$ that belong to $X$. If the inequality $P \le Q$ is not satisfied then $[P,Q]_X$ is empty.

In sections \ref{sec:Hardy}--\ref{sec:Finite} we avoid making any unnecessary assumptions, therefore that not all considered objects are means. Nevertheless, by the virtue of \eqref{ITSrest}, one can always add more assumptions, whenever it is required.

\subsection{Gini means}

Gini means were first considered by Gini \cite{Gin38} as a generalization of Power means. For $p,q \in \R$ and all-positive-entry vector $v=(v_1,\dots,v_n)$, $n \in \N$ they equal
\Eq{*}{
\G{p,q} (v_1,\dots,v_n)=\begin{cases}
\left(\frac{v_1^p+\cdots+v_n^p}{v_1^q+\cdots+v_n^q}\right)^{1/(p-q)} & \textrm{if\ }p \ne q\,,
\\
\exp \left(\frac{v_1^p \ln v_1+\cdots+v_n^p\ln v_n}{v_1^p+\cdots+v_n^p}\right)& \textrm{if\ }p = q\,.
                        \end{cases}
}

It is just a simple calculation that $\G{p,q}=\G{q,p}$. From our point of view the order in this family is the most important one. It is known (cf. \cite{DarLos70}) that for all $p,q,p',q' \in \R$,
\Eq{Ginicomp}{
\G{p,q} \le \G{p',q'} \iff \min(p,q) \le \min(p',q') \wedge \max(p,q) \le \max(p',q').
}
Therefore we can leave the definition of means itself and deal with the order $\prec$ on $\R^2$ defined exactly like the right hand side in \eqref{Ginicomp}. However it is difficult to characterize all interval-type sets in this family. We will focus just on the simplest case.

\begin{prop}
\label{Prop:Gini_Interval}
Let $p,q,r,s \in \R$, $p\le s$, and $q,r \in [p,s]$. Then $\G{p,q} \le \G{r,s}$ and 
 \Eq{Gini_Interval}
 {
 [\G{p,q},\G{r,s}]_{\G{}}=\{\G{x,y} \colon (x,y) \in [p,r]\times [q,s] \cup [q,s] \times [p,r]\}.
 }
\end{prop}

The assumption involving $p,q,r,s$ seamed to be restrictive, nevertheless it appears very naturally when we take into account the requirement of comparability between $\G{p,q}$ and $\G{r,s}$ and the mentioned symmetry $\G{p,q}=\G{q,p}$. 

To prove this proposition we need to consider two cases $p \le q \le r \le s$ and $p \le r < q \le s$. However, by the virtue of \eqref{Ginicomp}, this proof is elementary and we will omit it. 

Let us present another kind of interval-type sets for Gini means
\begin{prop}
Let $f \colon (0,\infty) \to (0,\infty)$, $g \colon (-\infty,0) \to (-\infty,0)$ be two continuous, decreasing functions with $f(0^+)=+\infty$, $g(0^-)=-\infty$ and $f \circ f=\id$, $g \circ g=\id$.
The following sets are of $\G{}$-interval-type
\Eq{*}{
\Big\{\G{x,y} &\colon x\le 0 
\text{ or }\big(  x >0 \text{ and } y \le f(x)\big) \Big\}, \\
\Big\{\G{x,y} &\colon x \ge 0 \text{ or } \big( x<0 \text{ and } y \ge g(x) \big) \Big\}, \\
\Big\{\G{x,y} &\colon \big( x<0 \text{ and } y \ge g(x) \big) \text{ or } x=0  \text{ or } \big( x >0 \text{ and } y \le f(x) \big) \Big\}.
}
\end{prop}

\begin{proof}
 Denote these sets by $X$, $Y$ and $X \cap Y$, respectively. Suppose that $\G{p,q}, \G{r,s} \in X$ for some $p,q,r,s$. We can assume without loss of generality that $p\le q$, $r\le s$ and $q \le s$. As these means are comparable, we immediately obtain the inequality $p \le r$. By Proposition~\ref{Prop:Gini_Interval}
  we need to prove that
  \Eq{ENonint_1}{
\{\G{x,y} \colon (x,y) \in [p,r]\times [q,s] \cup [q,s] \times [p,r]\}=[\G{p,q},\G{r,s}]_{\G{}} \subset X.
}
As $X$ covers {\it all} parameters in second, third, and fourth quarter  we can assume that all $p,q,r,s$ are non-negative.
Due to the fact that $\G{r,s} \in X$ we get $s \le f(r)$ or, equivalently (by the assumptions) $r \le f(s)$.
Now, for $(x_0,y_0) \in [p,r]\times [q,s]$, by decreasingness of $f$ we obtain
\Eq{*}{
y_0 \le s \le f(r) \le f(x_0).
}
Similarly for $(x_0,y_0) \in [q,s] \times [p,r]$ we get $y_0 \le r \le f(s) \le f(x_0)$, which
provides \eqref{ENonint_1}. The second case is analogous, while the third one is just a combining of first two cases.
\end{proof}

\subsection{\label{sec:Hardy}Hardy means}

Let us now consider so-called Hardy property of Means. Let $M \colon \bigcup_{i=1}^n \R_{+} \rightarrow \R_{+}$ be a mean. Let $H_M$ be the smallest extended real number satisfying
\Eq{*}
{
 \sum_{n=1}^\infty M(v_1,\dots,v_n) \le H_M \sum_{i=1}^\infty v_n \text{ for all } v \in \ell^1(\R_+).
}
We call $M$ to be the \emph{Hardy mean} if $H_M<+\infty$; the number $H_M$ is called the \emph{Hardy constant} of $M$. 
The definition of Hardy means was first introduced by P\'ales and Persson in \cite{PalPer04} but it was developed since 1920s, when Hardy constants for Power means were given in a series of papers \cite{Har20,Lan21,Car32,Har25a,Kno28}; more details about interesting history of this result can be found in catching surveys \cite{PecSto01,DucMcG03} and in a recent book \cite{KufMalPer07}. Term Hardy constant was introduced recently in \cite{PalPas16}.

Let us consider a family of sets 
\Eq{*}{
H_{\cdot}^{-1}(I):=\{M \colon H_M \in I\} \text{ for }I \subset (0,+\infty]
}
and the family of all Hardy means $\mathcal{H}$, that is $\mathcal{H}:=H_{\cdot}^{-1}(0,+\infty)$.

We are going to prove the following statement
\begin{prop}
For every (extended) interval $I \subset [0,+\infty]$ the family $H_{\cdot}^{-1}(I)$ is an interval-type set in a family of all means. In particular $\mathcal{H}$ is an interval-type set.
\end{prop}

\begin{proof}
 Take any mean $M$ and $P,\,Q$ satisfying $P\le M \le Q$.
It is sufficient to prove that 
 \Eq{Har_bas}
 {
 H_M\in [H_P,H_Q].
 }
 Indeed, for all $v \in \ell_1(\R_+)$ we have
 \Eq{*}
 {
  \sum_{n=1}^\infty M(v_1,\dots,v_n) \le   \sum_{n=1}^\infty Q(v_1,\dots,v_n) \le H_Q \sum_{i=1}^\infty v_n.
 }
By the definition $H_M$ is the smallest constant satisfying such an inequality thus $H_M \le H_Q$.
Analogously we can prove $H_P \le H_M$, which implies \eqref{Har_bas}.

 \end{proof}

\subsection{\label{sec:Finite}Finite distance between means}

For the purpose of this section $I$ is an interval and a domain of mean is an arbitrary set $\mathcal{D}$, that is $M$ is any function $M \colon \mathcal{D} \to I$ (it is natural to denote this family as $I^\mathcal{D}$), in particular all means are defined on {\it the same} domain.

In the spirit of \cite{CarShi64,CarShi69} let us define a metric among quasi-arithmetic means as a maximal possible difference between their value. This definition could be easily adapted for a general case. More precisely, for $M,N \in I^\mathcal{D}$ we define the distance as an $L^\infty$-norm,
\Eq{*}{
\rho(M,N):=\sup_{x \in \mathcal{D}} \abs{M(x)-N(x)}.
}
It is easy to prove that $\rho$ is an extended metric on the space $I^\mathcal{D}$. Therefore we can define neighborhoods of functions (each function is treated as a single point in the space) in the following way
\Eq{*}{
B_r(M)&:=\{N \in I^\mathcal{D} \colon \rho(M,N)<r\},\\
\overline{B}_r(M)&:=\{N \in I^\mathcal{D} \colon \rho(M,N)\le r\}.
}
If the set $I$ is unbounded, it can happen that the distance between means is not bounded i.e. $\rho(M,N)=+\infty$. Therefore it is also reasonable to define
\Eq{*}{
B(M):=\{N \in I^\mathcal{D} \colon \rho(M,N)<+\infty\}=\bigcup_{r>0} B_r(M).
 }
 
 Having these definitions there holds the following
 
\begin{prop}
 For each mean $M$ the set $B(M)$ is an $I^\mathcal{D}$--interval-type family, moreover for every positive number $r$, sets $B_r(M)$ and $\overline{B}_r(M)$ are $I^\mathcal{D}$--interval-type families. 
\end{prop}

The proof of this proposition is elementary and we omit it. 

\def\cprime{$'$} \def\R{\mathbb R} \def\Z{\mathbb Z} \def\Q{\mathbb Q}
  \def\C{\mathbb C}

\end{document}